\documentclass[a4paper,twoside,reqno]{amsart}

\usepackage[pagewise]{lineno}
\usepackage[utf8]{inputenc}

\usepackage{authblk}
\usepackage{hyperref}
\hypersetup{urlcolor=blue, citecolor=red}
\usepackage{amsmath,amsthm,amssymb,xcolor,bbm,mathrsfs}
\usepackage[english]{babel}
\usepackage{fancyhdr}
\newcommand{\norm}[1]{\|#1\|}

\makeatletter
\@namedef{subjclassname@1991}{2020 Mathematics Subject Classification}
\makeatother

\title{Mixed Hegselmann-Krause Dynamics\\ on infinite graphs}
\author{Hsin-Lun Li}

\date{}
\email{hsinlunl@asu.edu}

\theoremstyle{definition}
\newtheorem{theorem}{Theorem}

\newtheorem{lemma}[theorem]{Lemma}

\begin{document}

\allowdisplaybreaks

\thispagestyle{firstpage}
\maketitle
\begin{center}
    Hsin-Lun Li
\end{center}

\begin{abstract}
    The mixed Hegselmann-Krause (HK) model covers the synchronous Hegselmann-Krause model, the asynchronous Hegselmann-Krause model and the Deffuant model. Previous study~\cite{mHK, mHK2} deals with the mixed HK model on finite graphs. In the study, we discuss the mixed HK model on infinite graphs which also covers the HK model and the Deffuant model on infinite graphs. We investigate conditions under which asymptotic stability holds or any two vertices in the same component approaches each other after some finite time.
\end{abstract}

\section{Introduction}
The Hegselmann-Krause (HK) model and the Deffuant model are well-known confidence threshold models in continuous opinion dynamics. Both set a confidence threshold $\epsilon>0$ to distinguish whether two individuals are opinion neighbors or not. Two individuals are opinion neighbors if their opinion distance is up to $\epsilon.$ In certain discrete opinion models, such as the threshold voter model discussed in~\cite{andjel1992clustering, durrett1992multicolor, durrett1993fixation, liggett1994coexistence}, individuals adhere to a threshold, denoted by $\tau$. At this threshold, they update their opinions at a rate of one in continuous time and change their opinion if and only if the number of opponents in their neighborhood is at least $\tau$. The model proposed by the author in~\cite{li2024imitation} is akin to the threshold voter model, with the distinction that individuals change their opinions only when the number of opponents in their neighborhood constitutes the majority. In the HK model, an individual updates its opinion by taking the average opinion of its opinion neighbors. In the synchronous HK model, all individuals update their opinion at each time step. In contrast, in the asynchronous model, only one uniformly selected individual updates its opinion at each time step. In the Deffuant model, a pair of socially connected individuals are uniformly chosen at each time step and approach each other equally at a rate $\mu\in (0,1/2]$ if and only if their opinion distance is up to $\epsilon.$ The author in~\cite{mHK} proposed the mixed HK model and argued that it covers the synchronous HK model. In the sequel~\cite{mHK2}, the author further argued that it covers the asynchronous HK model and the Deffuant model. There are two modes for the mixed HK model: pair interaction and group interaction. For pair interaction, the update depends on interacting pairs, whereas for group interaction, the update depends on the non-absolutely stubborn individuals. We previously studied the mixed HK model on finite graphs. Now, we study the mixed HK model on infinite graphs. Namely, the number of individuals is infinite. Since individuals are countable, we set the vertex set as $\mathbb{Z}^+.$  Inheriting the spirit of the mixed HK model on finite graphs, we have the following definitions for mixed HK model on infinite graphs. A \emph{social graph} $G=(\mathbb{Z}^+, E)$ is an undirected graph with vertex set and edge set
$$\mathbb{Z}^+\ \hbox{and}\ E=\{(i,j)\in (\mathbb{Z}^+)^2: i\neq j\ \hbox{and}\   \hbox{vertices $i$ and $j$ are socially connected}\},$$ and all vertices are of finite degree. This indicates that all individuals have a finite number of social neighbors. A \emph{social graph for update} at time $t$, $\Tilde{G}(t)=(\mathbb{Z}^+,\Tilde{E}(t))$, is a subgraph of the social graph for the opinion update. An \emph{opinion graph} at time $t$, $\mathscr{G}(t)=(\mathbb{Z}^+,\mathscr{E}(t))$, is an undirected graph with vertex set and edge set
$$\mathbb{Z}^+\ \hbox{and}\ \mathscr{E}(t)=\{(i,j)\in (\mathbb{Z}^+)^2: i\neq j\ \hbox{and}\ \norm{x_i(t)-x_j(t)}\leq \epsilon\}.$$ An opinion graph $\mathscr{G}$ is \emph{$\delta$-trivial} if any two vertices in $\mathscr{G}$ are at a distance of at most $\delta$ apart. A \emph{profile} at time $t$, $\tilde{G}(t)\cap \mathscr{G}(t)$, is the intersection of the social graph for update and the opinion graph at time $t$. Although an individual may have an infinite number of opinion neighbors, the number of social and opinion neighbors of each individual is finite. Unlike the Deffuant model where only one pair of socially connected agents are selected at each time step, the pair interaction of the mixed HK model allows several pairs of socially connected individuals chosen at each time step. A \emph{matching} $M$ in graph $G$ is a set of pairwise nonadjacent edges, none of which are loops. The collection of several pairs of socially connected individuals corresponds to a matching in social graph $G.$ In the mixed HK model, each individual can decide its degree of stubbornness and mix its opinion with the average opinion of its social and opinion neighbors for opinion update. $\alpha_i(t)$ is the degree of stubbornness of individual $i$ at time $t.$ An individual is \emph{absolutely stubborn} if its degree of stubbornness is 1. $x_i(t)\in\mathbb{R}^d$ is the opinion of individual $i$ at time $t.$ Denote $\|\alpha\|\leq \|\beta\|$ as $\alpha\ll \beta$. We assume that 
\begin{itemize}
    \item $\alpha_i(t)$, $t\geq 0$ are independent and identically distributed random variables on $[0,1]$ for all $i\geq 1$, \vspace{2pt}
    \item $U_t$, $t\geq 0$ are independent and identically distributed random variables with a support $S$. For the pair interaction, $U_t=\{(i,j)\in (\mathbb{Z}^+)^2: i\neq j\ \hbox{and}\ ( \alpha_i(t)<1\ \hbox{or}\ \alpha_j(t)<1 )\}$ and $S\subset \{\hbox{all matchings in}\ (\mathbb{Z}^+)^2\}$, whereas for the group interaction, $U_t=\{i\geq 1:\alpha_i(t)<1\}$, $S\subset \mathscr{P}(\mathbb{Z}^+)$, the power set of $\mathbb{Z}^+$, and $V(S)=\{i: i\in a\ \hbox{for some}\ a\in S\}\supset \mathbb{Z}^+$, and \vspace{2pt}
    \item $\Tilde{E}(t)=U_t\cap E$ for the pair interaction, whereas $\tilde{E}(t)=E$ for the group interaction.
\end{itemize}
The evolution rule of the mixed HK model is as follows:
\begin{equation}\label{Eq: mHK}
x_i(t+1)=\alpha_i(t)x_i(t)+\frac{1-\alpha_i(t)}{|\mathscr{N}_i(t)|}\sum_{k\in \mathscr{N}_i(t)}x_k(t)   
\end{equation}
 where 
$\mathscr{N}_i(t)=\{j\geq 1:j=i\ \hbox{or}\  (i,j)\in \tilde{E}(t)\cap \mathscr{E}(t)\}$ is the collection of social and opinion neighbors of agent $i$ for opinion update. In particular, \eqref{Eq: mHK} reduces to
\begin{itemize}
    \item the synchronous HK model on infinite graphs if $S=\{\mathbb{Z}^+\}$ and $\alpha_i(t)=0$ for all $i\geq 1$ and $t\geq 0$, \vspace{2pt} 
    \item the asynchronous HK model on infinite graphs if $S=\big\{\{i\}\big\}_{i\geq 1}$ and $\alpha_i(t)=0$ if $i\in U_t$ at all times, and \vspace{2pt}
    \item the Deffuant model on infinite graphs if $S=\big\{\{(i,j)\}\big\}_{(i,j)\in E}$ and $\alpha_i(t)=\alpha_j(t)=1-2\mu$ for all $(i,j)\in U_t$ and $t\geq 0$.
\end{itemize}
Some properties of the mixed HK model on finite graphs do not always hold on infinite graphs. In~\cite{mHK2}, for finite social graph $G$ under group interaction, there is a consensus if a profile is connected infinitely many times and the most stubborn individual of the non-absolutely stubborn over time is non-absolutely stubborn. Since $\mathbb{Z}$ is equivalent to $\mathbb{Z}^+$ on countability, consider social graph $G$ an infinite cycle on $\mathbb{Z}$ with edges $(i,i+1)$, $x_i(0)=i$ and threshold $\epsilon=1$ for all $i\in \mathbb{Z}$. It turns out that $x_i=i$ all the time so there is no consensus. The works by \cite{lanchier2020probability} and \cite{lanchier2022consensus} delved into the probability of consensus in the Deffuant model on finite graphs and the Hegselmann-Krause model on finite graphs, respectively. Additionally, the author in~\cite{MR3069370} established that the critical value of the Deffuant model is one half.


\section{Main results}
 Under group interaction, the profile $\Tilde{G}(t)\cap\mathscr{G}(t)$ equals $G\cap \mathscr{G}(t).$ In contrast, under pair interaction, the edge set of a profile is a matching in $G.$ Let $\Delta x_i(t)=x_{i+1}(t)-x_i(t),\ i\geq 1$ be first differences of $x_i(t),\ i\geq 1.$ For a social graph $G$ representing a path under group interaction in one-dimensional space, if a profile equals $G$ at some time, where $G$ corresponds to nonincreasing opinions and nondecreasing first differences, then all opinions are asymptotically stable. 

\begin{theorem}\label{Thm: 1-dim asymptotic stability}
    For group interaction in one-dimensional space, assume that the social graph $G$ is a path with edges $(i,i+1)$ for all $i\geq 1$. If profile $G\cap \mathscr{G}(s)$ equals $G$ at some time~$s$ with $(x_i(s))_{i\geq 1}$ nonincreasing and $(\Delta x_i(s))_{i\geq 1}$ nondecreasing, $x_i$ is asymptotically stable for all $i\geq 1$.
\end{theorem}
Since a profile on infinite graphs still preserves $\delta$-triviality, profile $G\cap \mathscr{G}(t)$ equals $G$ for all $t\geq s$ if opinion graph $\mathscr{G}(s)$ is $\epsilon$-trivial for some $s\geq 0$. Let $N_i^s=\{j\geq 1: j=i\ \hbox{or}\ (i,j)\in E\}$ be the collection of all social neighbors of $i$ including itself. For the social graph $G$ $(r-1)$-regular for some $r\geq 1$ under group interaction, we show circumstances under which any two vertices in the same component of $G$ approach each other after some finite time. $\alpha_i(t)=\alpha_t$ for all $i\geq 1$ and $t\geq 0$ indicates all individuals have the same degree of stubbornness at each time step. $\liminf_{t\to\infty}\alpha_t<1$ reveals that the degree of stubbornness is up to some number less than one infinitely many times. $3|N_i^s\cap N_j^s|> 2r$ for all edges $(i,j)\in E$ represents the number of common vertices of adjacent vertices $i$ and $j$ including themselves is greater than $2r/3.$

\begin{theorem}\label{Thm: consensus of mHK under group interaction}
     For group interaction, assume that social graph $G$ is $(r-1)$-regular for some $r\geq 1$, $\alpha_i(t)=\alpha_t$ for all $i\geq 1$ and $t\geq 0$ with $\liminf_{t\to\infty}\alpha_t<1$ and $3|N_i^s\cap N_j^s|> 2r$ for all edges $(i,j)\in E$. If profile $G\cap\mathscr{G}(s)$ is $G$ for some $s\geq 0$, $\sup_{(i,j)\in E}\|x_i(t)-x_j(t)\|\to 0$ as $t\to\infty.$
\end{theorem}
This implies $\|x_i(t)-x_j(t)\|\to 0$ as $t\to\infty$ for all $i,j\geq 1$ if $G$ is connected. For social graph $G$ connected under pair interaction, we show conditions under which any two vertices are close to each other after some finite time. $\bigcup_{a\in S}a\supset E$ indicates all pairs of socially connected individuals are selected with positive probability. $0\leq \sup_{t\geq 0} \sup \{\alpha_i (t):i\geq 1\ \hbox{and}\ \alpha_i(t)<1\}< 1$ represents that the most stubborn individual of the non-absolutely stubborn over time is non-absolutely stubborn.

\begin{theorem}\label{Thm: consensus of mHK under pair interaction}
    For pair interaction, assume that social graph $G$ is connected, all interacting pairs approach each other equally at their rate, $\bigcup_{a\in S}a\supset E$ and that $$0\leq \sup_{t\geq 0} \sup \{\alpha_i (t):i\geq 1\ \hbox{and}\ \alpha_i(t)<1\}< 1.$$ Then, $\|x_i(t)-x_j(t)\|\to 0$ as $t\to\infty$ for all $i,j\geq 1.$
\end{theorem}

\section{Properties of the model}
For the mixed HK model on finite graphs, a profile preserves $\delta$-triviality. It follows from Lemma~\ref{matching} that a profile still preserves $\delta$-triviality.
\begin{lemma}[\cite{mHK}]\label{matching}
Given $\lambda_1,...,\lambda_n$ in $\mathbf{R}$ with $\sum_{i=1}^n\lambda_i=0$ and $x_1,...,x_n$ in $\mathbf{R^d}$. Then for $\lambda_1x_1+\lambda_2x_2+...+\lambda_nx_n,$ the terms with positive coefficients can be matched with the terms with negative coefficients in the sense that there are nonnegative values $c_i$ such that
\begin{align*}
     \sum_{i=1}^n\lambda_ix_i=\sum_{i,c_i\geq0,j,k\in[n]} c_{i}(x_{j}-x_{k})\text{ and }\sum_i c_i=\sum_{j,\lambda_j\geq0}\lambda_j.
\end{align*}
\end{lemma}
 For social graph $G$ a path under group interaction, a profile remains to be a path if it is $G$ at some time.

\begin{lemma}[path preserving]\label{path preserving}
    Assume that social graph $G$ is a path under group interaction. If profile $G\cap\mathscr{G}(t)$ equals $G$, so does profile $G\cap\mathscr{G}(t+1)$ for all $t\geq 0$.
\end{lemma}

\begin{proof}
    Without loss of generality, let $G$ be the path with edges $(i,i+1)$ for all $i\geq 1.$ Profile $G\cap\mathscr{G}(t)=G$ implies edge $(i,i+1)\in \mathscr{E}(t)$ for all $i\geq 1.$ We claim that edge $(i,i+1)\in \mathscr{E}(t+1)$ for all $i\geq 1$. Let $\alpha_i=\alpha_i(t)$, $x_i=x_i(t)$ and $x_i^\star=x_i(t+1)$ for all $i\geq 1$ and $t\geq 0.$  Since edges $(i,i+1),\ i\geq 2$ are equivalent, we see if edge $(2,3)\in \mathscr{E}(t+1).$ By the triangle inequality, we obtain
    \begin{align}
         x_1^\star-x_2^\star& =\alpha_1x_1+(1-\alpha_1)\frac{x_1+x_2}{2}-\big[\alpha_2 x_2+(1-\alpha_2)\frac{x_1+x_2+x_3}{3}\big]\notag\\
        &=\frac{1+3\alpha_1+2\alpha_2}{6}x_1+\frac{1-3\alpha_1-4\alpha_2}{6}x_2-\frac{1-\alpha_2}{3}x_3 \label{Edge12}\\
         x_2^\star-x_{3}^\star&=\alpha_2 x_2+(1-\alpha_2)\frac{x_{1}+x_2+x_{3}}{3}-\big[\alpha_3 x_3+(1-\alpha_3)\frac{x_2+x_{3}+x_{4}}{3}\big]\notag\\
        &=\frac{1-\alpha_2}{3}x_1+\frac{2\alpha_2+\alpha_3}{3}x_2-\frac{\alpha_2+2\alpha_3}{3}x_3-\frac{1-\alpha_3}{3}x_4 \label{Edge23}
    \end{align}
    
    \begin{align}
      \eqref{Edge12}=  &\bigg\{\begin{array}{lc}
       \displaystyle \frac{1+3\alpha_1+2\alpha_2}{6}(x_1-x_3)+\frac{1-3\alpha_1-4\alpha_2}{6}(x_2-x_3) & \hbox{if}\quad 3\alpha_1+4\alpha_2<1\vspace{2pt}\notag\\
       \displaystyle -\frac{1-3\alpha_1-4\alpha_2}{6}(x_1-x_2)+\frac{1-\alpha_2}{3}(x_1-x_3) &\hbox{otherwise,}
       \end{array}\notag\\
       \ll&\bigg\{\begin{array}{lc}
       \displaystyle \big(\frac{1+3\alpha_1+2\alpha_2}{6}\cdot 2+\frac{1-3\alpha_1-4\alpha_2}{6}\big)\sup_{(i,j)\in E}\|x_i-x_j\| \vspace{2pt}\notag\\
       \displaystyle \big(-\frac{1-3\alpha_1-4\alpha_2}{6}+\frac{1-\alpha_2}{3}\cdot 2\big) \sup_{(i,j)\in E}\|x_i-x_j\|
       \end{array}\notag\\
       =&\frac{1+\alpha_1}{2}\sup_{(i,j)\in E}\|x_i-x_j\|\ll\epsilon\notag\\
       \eqref{Edge23}=&\bigg\{\begin{array}{lc}
       \displaystyle \frac{1-\alpha_2}{3}(x_1-x_4)+\frac{\alpha_2-\alpha_3}{3}(x_2-x_4)+\frac{\alpha_2+2\alpha_3}{3}(x_2-x_3) & \hbox{if}\quad \alpha_2\geq\alpha_3\vspace{2pt}\\
       \displaystyle\frac{1-\alpha_3}{3}(x_1-x_4)+\frac{\alpha_3-\alpha_2}{3}(x_1-x_3)+\frac{\alpha_3+2\alpha_2}{3}(x_2-x_3) &\hbox{otherwise,}
       \end{array}\notag\\
       \ll&\bigg\{\begin{array}{lc}
       \displaystyle \big(\frac{1-\alpha_2}{3}\cdot 3+\frac{\alpha_2-\alpha_3}{3}\cdot 2+\frac{\alpha_2+2\alpha_3}{3}\big) \sup_{(i,j)\in E}\|x_i-x_j\| \vspace{2pt}\notag\\
       \displaystyle \big(\frac{1-\alpha_3}{3}\cdot 3+\frac{\alpha_3-\alpha_2}{3}\cdot 2+\frac{\alpha_3+2\alpha_2}{3}\big) \sup_{(i,j)\in E}\|x_i-x_j\|
       \end{array}\notag\\
      = & \sup_{(i,j)\in E}\|x_i-x_j\|\ll \epsilon.\notag
    \end{align}
     So edge $(i,i+1)\in \mathscr{E}(t+1)$ for all $i\geq 1$, which implies profile $G\cap\mathscr{G}(t+1)$ equals $G.$
\end{proof}
In particular for group interaction in one-dimensional space, the opinion order preserves if social graph $G$ is a path corresponding to monotone opinions and the profile equals $G$ at some time.

\begin{lemma}[order preserving]\label{order preserving}
    For group interaction in one-dimensional space, assume that the social graph $G$ is the path with edges $(i,i+1)$ for all $i\geq 1.$ If profile $G\cap\mathscr{G}(t)$ equals $G$ with $(x_i(t))_{i\geq 1}$ nondecreasing, so does profile $G\cap\mathscr{G}(t+1)$ for all $t\geq 0$. 
\end{lemma}

\begin{proof}
    Let $x_i=x_i(t)$ and $x_i^\star=x_i(t+1)$ for all $i\geq 1$ and $t\geq 0.$ Via Lemma~\ref{path preserving}, profile $G\cap\mathscr{G}(t+1)$ equals $G.$ It follows from~\eqref{Edge12} and~\eqref{Edge23}  that $x_1^\star-x_2^\star\leq 0$ and $x_{i}^\star-x_{i+1}^\star\leq 0$ for all $i>1$ so $(x_i(t+1))_{i\geq 1}$ is nondecreasing. 
\end{proof}
For a social graph $G$ $(r-1)$-regular for some $r\geq 1$ under group interaction, if all individuals share the same degree of stubbornness at all times and the number of common social neighbors of each socially connected pair including themselves is at least $2r/3$, a profile preserves regularity if it is $G$ at some time.

\begin{lemma}\label{regular preserving}
    For group interaction, assume that social graph $G$ is $(r-1)$-regular for some $r\geq 1$, $\alpha_i(t)=\alpha_t$ for all $i\geq 1$ and $t\geq 0$, and $|N_i^s\cap N_j^s|\geq 2r/3$ for all edges $(i,j)\in E$. If profile $G\cap\mathscr{G}(t)$ is $G$, so is profile $G\cap\mathscr{G}(t+1)$ for all $t\geq 0$. In particular,
    \begin{align*}
         \sup_{(i,j)\in E}\|x_i(t+1)-x_j(t+1)\|&\\
        &\hspace{-3cm}\leq \frac{\sup_{(i,j)\in E}\|x_i(t)-x_j(t)\|}{r}\big[r-(3\inf_{(i,j)\in E}|N_i^s\cap N_j^s|-2r)(1-\alpha_t)\big].
    \end{align*}
\end{lemma}

\begin{proof}
    The profile is $G$ at all times when $r=1.$ For $r>1$, let $x_i^\star=x_i(t+1)$ and $x_i=x_i(t)$ for all $i\geq 1$ and $t\geq 0$. For edge $(i,j)\in E\cap \mathscr{E}(t)=E$ and $t\geq 0$, by the triangle inequality, we get
    \begin{align*}
        x_i^\star-x_j^\star=&\alpha_t x_i-\alpha_t x_j+\frac{1-\alpha_t}{r}\sum_{k\in N_i^s-N_j^s}x_k-\frac{1-\alpha_t}{r}\sum_{k\in N_j^s-N_i^s} x_k\\
        =&\frac{1-\alpha_t}{r}\sum_{1\leq k\leq r-|N_i^s \cap N_j^s|;\atop p_k\in N_i^s-N_j^s;\ q_k\in N_j^s-N_i^s }(x_{p_k}-x_{q_k}) +\alpha_t(x_i-x_j)\\
        \ll & \frac{\sup_{(i,j)\in E}\|x_i-x_j\|}{r}\big[ r\alpha_t+3(1-\alpha_t)(r-|N_i^s \cap N_j^s|)\big ]\\
        \ll & \frac{\sup_{(i,j)\in E}\|x_i-x_j\|}{r}\big[r-(3\inf_{(i,j)\in E}|N_i^s\cap N_j^s|-2r)(1-\alpha_t)\big]\ll \epsilon.
    \end{align*}
    Hence, edge $(i,j)\in E\cap \mathscr{E}(t+1)$ and profile $G\cap\mathscr{G}(t+1)$ is $G$.
\end{proof}
Lemma~\ref{supermartingale for a finite number of agents} describes nonnegative supermartingales for group interaction and pair interaction if all interacting pairs approach each other equally at their rate on finite graphs. This helps us derive nonnegative supermartingales for group interaction and pair interaction if all interacting pairs approach each other equally at their rate on infinite graphs.

\begin{lemma}[\cite{mHK}]\label{supermartingale for a finite number of agents}
    Let \begin{equation*}
        \begin{array}{l}
        \displaystyle    Z_p(t)=\sum_{i,j\in [n]}\|x_i(t)-x_j(t)\|^2,  \\
        \displaystyle    Z_g(t)=\sum_{i,j\in [n]}(\|x_i(t)-x_j(t)\|^2\wedge\epsilon^2)\vee \epsilon^2\mathbbm{1}\{(i,j)\notin E\}. 
        \end{array}
    \end{equation*} In a system of $n$ individuals, we have $$Z_p(t)-Z_p(t+1)\geq 4\sum_{i\in [n]}\|x_i(t)-x_i(t+1)\|^2$$ for pair interaction if all interacting pairs approach each other equally at their rate, and $$Z_g(t)-Z_g(t+1)\geq 4\sum_{i\in [n]}\|x_i(t)-x_i(t+1)\|^2$$ for group interaction.
    
\end{lemma}

Letting $n\to\infty$ on 
$$\begin{array}{l}
    \displaystyle Z_p(t)\geq Z_p(t+1)+4\sum_{i\in [n]}\|x_i(t)-x_i(t+1)\|^2,  \\
    \displaystyle Z_g(t)\geq Z_p(t+1)+4\sum_{i\in [n]}\|x_i(t)-x_i(t+1)\|^2, 
\end{array}$$ we derive Lemma~\ref{supermartingale for countable individuals}. 
\begin{lemma}\label{supermartingale for countable individuals}
    In a system of countable individuals, there is a nonnegative function $Z_t$ nonincreasing to $t$ such that $$Z_t\geq Z_{t+1}+4\sum_{i\geq 1}\|x_i(t)-x_i(t+1)\|^2$$ for group interaction, and pair interaction if all interacting pairs approach each other equally at their rate.
\end{lemma}
With the nonnegative supermartingales on infinite graphs, we obtain $x_i(t)-x_i(t+1)\to 0$ as $t\to\infty$ for all $i\geq 1$ for group interaction and pair interaction if all interacting pairs approach each other equally at their rate in Lemma~\ref{Differences}.

\begin{lemma}\label{Differences}
    We get $x_i(t)-x_i(t+1)\to 0$ as $t\to\infty$ for all $i\geq 1$ for pair interaction if all interacting pairs approach each other equally at their rate, and group interaction.
\end{lemma}

\begin{proof}
    It follows from Lemma~\ref{supermartingale for countable individuals} that there is a nonnegative supermartingale $Z_t$ with 
    \begin{equation}\label{eq: supermartingale}
        Z_t\geq Z_{t+1}+\sum_{i\geq 1}\|x_i(t)-x_i(t+1)\|^2.
    \end{equation} By the martingale convergence theorem, $Z_t$ approaches a random variable $Z_\infty$ with finite expectation. Applying $t\to\infty$ on~\eqref{eq: supermartingale}, we derive $x_i(t)-x_i(t+1)\to 0$ as $t\to\infty$ for all $i\geq 1.$ 
\end{proof}

\begin{proof}[\bf Proof of Theorem~\ref{Thm: 1-dim asymptotic stability}]
Without loss of generality, assume that $(x_i(s))_{i\geq 1}$ is nonincreasing. Via Lemma~\ref{order preserving}, $(x_i(t))_{i\geq 1}$ is nonincreasing for all $t\geq s$. For $t\geq s$ and $i> 1$, we obtain
\begin{align*}
    x_1(t+1)-x_1(t)&=(1-\alpha_1(t))\frac{x_2(t)-x_1(t)}{2}\leq 0,\\
    x_i(t+1)-x_i(t)&=(1-\alpha_i(t))\big(\frac{x_{i-1}(t)-x_i(t)}{3}+\frac{x_{i+1}(t)-x_i(t)}{3}\big)\\
    &=\frac{1-\alpha_i(t)}{3}[-\Delta x_{i-1}(t)+\Delta x_i(t)]\geq 0.
\end{align*}
It turns out from Lemma~\ref{order preserving} that $x_1(t)\geq x_2(t)\geq x_2(s)$ and $x_i(t)\leq x_1(t)\leq x_1(s)$ for all $i> 1.$ Since $(x_1(t))_{t\geq s}$ is nonincreasing and bounded below by $x_2(s)$ and $(x_i(t))_{t\geq s}$ is nondecreasing bounded above by $x_1(s)$ for all $i>1$, $x_i$ is asymptotically stable for all $i\geq 1.$ 
\end{proof}
Considering $x_i(0)=1/i$, threshold $\epsilon=1/2$ and social graph $G$ a path with edges $(i,i+1)$ for all $i\geq 1$, observe that profile $G\cap\mathscr{G}(0)=G$ corresponds to $(x_i(0))_{i\geq 1}$ nonincreasing and $(\Delta x_i(0))_{i\geq 1}$ nondecreasing. By Theorem~\ref{Thm: 1-dim asymptotic stability}, $x_i$ is asymptotically stable for all $i\geq 1.$ In particular, we obtain 
\begin{align*}
    &1/2=x_2(0)\leq\lim_{t\to\infty}x_1(t)\leq x_1(0)=1,\\
    &1/i=x_i(0)\leq\lim_{t\to\infty}x_i(t)\leq x_1(0)=1\ \hbox{for all}\ i>1.
\end{align*}
While zero serves as the limit point for all opinions at the initial time, none of their opinion limit falls on it.
\begin{proof}[\bf Proof of Theorem~\ref{Thm: consensus of mHK under group interaction}]
    $\liminf_{t\to\infty}\alpha_t<1$ implies $\limsup_{t\to\infty}(1-\alpha_t)>0$. There are $(t_k)_{k\geq 0}$ increasing with $t_0\geq s$ and $1>\delta>0$ such that $1-\alpha_{t_k}\geq \delta$ for all $k\geq 0.$ $3|N_i^s\cap N_j^s|> 2r$ implies $3|N_i^s\cap N_j^s|\geq 2r+1.$ Let $A_t=\sup_{(i,j)\in E}\|x_i(t)-x_j(t)\|$. For $t>s$, $t_{k}<t\leq t_{k+1}$ for some $k\geq 0.$ Via Lemma~\ref{regular preserving}, $A_t\leq (1-\delta/r)^k A_0.$ As $t\to\infty$, $k\to\infty.$ It follows that $\limsup_{t\to\infty} A_t\leq 0$, therefore $\lim_{t\to\infty}A_t=0.$
\end{proof}

\begin{proof}[\bf Proof of Theorem~\ref{Thm: consensus of mHK under pair interaction}]
    We claim that $\|x_i(t)-x_j(t)\|\to 0$ as $t\to\infty$ for all edges $(i,j)\in E.$ If this is not the case, there are $\delta>0$, $(t_k)_{k\geq 0}$ increasing and edge $(i,j)\in E$ such that $\|x_i(t_k)-x_j(t_k)\|\geq\delta$ for all $k\geq 0.$ Since $\bigcup_{a\in S}a\supset E$ and $$0\leq \sup_{t\geq 0} \sup \{\alpha_i (t):i\geq 1\ \hbox{and}\ \alpha_i(t)<1\}< 1,$$ by Borel-Cantelli lemma, there are $(s_k)_{k\geq 0}\subset (t_k)_{k\geq 0}$ and $1/2\geq \gamma>0$ with $\mu(s_k)\geq\gamma$ for all $k\geq 0$ such that $$x_i(s_k+1)-x_i(s_k)=\mu(s_k)(x_j(s_k)-x_i(s_k))\gg \gamma(x_j(s_k)-x_i(s_k)).$$ It follows from Lemma~\ref{Differences} that $x_j(s_k)-x_i(s_k)\to 0$ as $k\to\infty,$ contradicting $\|x_i(t_k)-x_j(t_k)\|\geq\delta$ for all $k\geq 0.$ For any two vertices in $G,$ there is a finite path connecting the two. By the triangle inequality, $\|x_i(t)-x_j(t)\|\to 0$ as $t\to\infty$ for all $i,j\geq 1.$
\end{proof}

\section{Statements and Declarations}
\subsection{Competing Interests}
The author is partially supported by NSTC grant.

\subsection{Data availability}
No associated data was used.

\end{document}